\documentclass[12pt]{article}
\usepackage{latexsym, amsmath, amsfonts, amsthm, amssymb}
\usepackage{times}
\usepackage{a4wide}
\usepackage{tikz,pgfplots}
\usepackage{mathrsfs}
\usetikzlibrary{arrows}

\def \RR {\mathbb R}

\def \ZZ {\mathbb Z}

\def \eps {\varepsilon}

\newtheorem{theorem}{Theorem}[section]

\newtheorem{lemma}[theorem]{Lemma}

\newtheorem{proposition}[theorem]{Proposition}
\newtheorem{corollary}[theorem]{Corollary}

\def\myffrac#1#2 in #3{\raise 2.6pt\hbox{$#3 #1$}\mkern-1.5mu\raise 0.8pt\hbox{$
		#3/$}\mkern-1.1mu\lower 1.5pt\hbox{$#3 #2$}}

\def\qed{\hfill $\vcenter{\hrule height .3mm
		\hbox {\vrule width .3mm height 2.1mm \kern 2mm \vrule width .3mm
			height 2.1mm} \hrule height .3mm}$ \bigskip}

\def \Im {{\rm  Im  }}

\begin{document}

\title{Euclidean nets under isometric embeddings}
\author{Matan Eilat}
\date{}
\maketitle

\begin{abstract} 
Suppose that there exists a discrete subset $X$ of a complete, connected, $n$-dimensional Riemannian manifold $M$ such that the Riemannian distances between points of $X$ correspond to the Euclidean distances of a net in $\RR^{n}$.
What can then be derived about the geometry of $M$? 
In \cite{EK} it was shown that if $n=2$ then $M$ is isometric to $\RR^{2}$.
In this paper we show two consequential geometric properties that the manifold $M$ shares with the Euclidean space in any dimension. 
The first property is that $X$ is a net with respect to the Riemannian distance in $M$.
The second property is that all geodesics in $M$ are distance minimizing, and there are no conjugate points in $M$. 
This demonstrates the possibility of inferring infinitesimal qualities from discrete data, even in higher dimensions.
As a corollary we obtain that the large-scale geometry of $M$ is asymptotically Euclidean. 
\end{abstract}

\section{Introduction}

Let $M$ be a complete, connected, $n$-dimensional Riemannian manifold.
We say that a metric space $(X,d_{X})$ embeds isometrically in $M$ if there exists a map $\iota : X \to M$ that preserves distances, i.e. such that
$$
d_{M}(\iota(x),\iota(y)) = d_{X}(x,y)
\qquad \text{for any } x,y \in X,
$$
where $d_{M}$ is the Riemannian distance function in $M$.
When given such an embedding, one can think about the metric space $X$ as a subset of the manifold $M$ for which the distances are ``known''.
The question is then what can we learn about the manifold itself from the knowledge of these distances. 
It appears that for some discrete metric spaces, this partial data can determine the topology, and sometimes even the whole geometry of the underlying manifold.

\medskip
A discrete set $X \subseteq \RR^{n}$ is called a $\delta$-net for $\delta > 0$, if $d(y,X) = \inf_{x \in X} |x - y| < \delta$ for any $y \in \RR^{n}$. 
According to \cite[Theorem 1.1]{EK} an isometric embedding of a net in the Euclidean plane into a two-dimensional surface determines its geometry completely:

\begin{theorem}
Let $M$ be a complete, connected, $2$-dimensional Riemannian manifold. Suppose that there exists a net in $\RR^2$ that embeds isometrically in $M$.
Then the manifold $M$ is flat and it is isometric to the Euclidean plane.
\label{thm_12170}
\end{theorem}

\medskip
It is still unknown whether the $n$-dimensional analog of this theorem holds true.
However, in any dimension the existence of such an embedding determines the {\it topology} of the underlying manifold.
If $X$ is a net in $\RR^{n}$ that embeds isometrically in $M$, then in \cite[Corollary 4.1]{EK} it is shown that all the geodesics passing through a point $p \in \iota(X)$ are distance minimizing. Thus the exponential map $\exp_{p}: T_{p}M \to M$ is a diffeomorphism, and in particular we obtain the following theorem. Note that in dimension $n=4$ there exist exotic smooth structures on $\RR^{4}$ according to \cite{Do, Fr}.
Theorem \ref{thm_1217_0} refers to the flat $\RR^{4}$ in this case.

\begin{theorem}
Let  $M$ be a complete, connected, $n$-dimensional Riemannian manifold.
Suppose that there exists a net in $\RR^n$ that embeds isometrically in $M$.
Then $M$ is diffeomorphic to $\RR^n$.
\label{thm_1217_0}
\end{theorem}

\medskip
We consider these results part of the study of metric rigidity. A well-known conjecture in this field is the boundary distance conjecture of Michel \cite{M}, which suggests that in a simple Riemannian manifold with boundary, the collection of distances between boundary points determines the Riemannian structure, up to an isometry. To date, Michel's conjecture has been proven only in two dimensions, by Pestov and Uhlmann \cite{PU}.
 
\medskip
In this paper we show that the manifold $M$ and the Euclidean space share additional indicative geometric properties in {\it any} dimension.
The following theorem states that the image of a net in Euclidean space under an isometric embedding is also a net with respect to the Riemannian distance in the manifold.

\begin{theorem}
Let $M$ be a complete, connected, $n$-dimensional Riemannian manifold. Suppose that $X \subseteq \RR^{n}$ is a $\delta$-net for some $\delta > 0$, and that $\iota : X \to M$ is an isometric embedding. Then $\iota(X)$ is a $(2 \delta n)$-net with respect to the Riemannian distance in $M$.
\label{thm_1730}	
\end{theorem}

\medskip
Loosely speaking, the statement in Theorem \ref{thm_1730} is about medium-scale geometry. It implies that there are no mesoscopic portions of the manifold that are oblivious of the embedding.
Our next theorem states that the manifold $M$ has the property that all geodesics are distance minimizing, and therefore $M$ is without conjugate points. 
Using informal jargon again, the theorem eliminates the possibility of any ``obstruction'' in the manifold, even a microscopic one.
It demonstrates the perhaps unexpected phenomenon of inferring infinitesimal properties from discrete data.

\begin{theorem}
Let $M$ be a complete, connected, $n$-dimensional Riemannian manifold. Suppose that there exists a net in $\RR^{n}$ that embeds isometrically in $M$. Then all the geodesics in $M$ are distance minimizing, and there are no conjugate points in $M$.
\label{thm_1645}
\end{theorem}

\medskip
A common practice in the mathematical literature is to use the absence of conjugate points as an {\it assumption} on the manifold.
This has been shown to have considerable implications in a variety of cases.
In 1948, Hopf \cite{H} showed that a two-dimensional Riemannian torus without conjugate points is flat, solving a question raised by Morse and Hedlund \cite{MH}. The latter proved the theorem under the additional assumption that there are no focal points.
The validity of the $n$-dimensional analog of Hopf's result was unknown for quite some time, with considerable contributions to the subject by Busemann \cite{Bu} and others, until it was finally answered in the affirmative by Burago and Ivanov \cite{BI} in 1994.
Further literature on the rigidity of Riemannian manifolds without conjugate points includes contributions by Burns and Knieper \cite{BK}, Bangert and Emmerich \cite{BE1, BE2}, Croke \cite{Cr1, Cr2} and Kleiner \cite{CK1, CK2}, Innami \cite{INN} and others.
Note that the statement of Theorem \ref{thm_1645} is different in this sense.
To the best of our knowledge, this is the first time that the non-existence of conjugate points has been {\it derived} from a rather straightforward hypothesis in dimension $n > 2$, which does not impose a-priori restrictions on the curvature tensor.

\medskip
The techniques used in \cite{EK} for the two-dimensional setting rely heavily on the topology of $\RR^{2}$, and in particular on the intersection of certain curves.
Here we develop a completely new topological framework that works for any dimension. 
This framework is outlined in Section \ref{section_phi} and is applied for the proof of Theorem \ref{thm_1730}.
In Section \ref{sec_2} we describe a notion in the spirit of the ``ideal boundary'' of a Hadamard manifold, which we use in Section \ref{sec_3} to present a key geometric idea for the proof of Theorem \ref{thm_1645}.
This idea is then implemented on top of the above-mentioned topological framework in order to finalize the proof of the theorem in Section \ref{sec_ncp}.
In Section \ref{sec_largescale} we discuss the large-scale geometry of the manifold $M$, and in particular we conclude that it tends to the Euclidean one.

\medskip
{\it Acknowledgements.} I would like to express my deep gratitude and appreciation to my advisor, Prof. Bo'az Klartag, for his dedicated guidance and ongoing support. I would also like to thank Prof. Lev Buhovsky for the inspiring discussions and valuable suggestions.
Supported by the Adams Fellowship Program of the Israel Academy of Sciences and Humanities and the Israel Science Foundation (ISF).
 
\section{The manifold as a continuous image of Euclidean space} \label{section_phi}
 
From now on, our standing assumptions are those of Theorem \ref{thm_1730} and Theorem \ref{thm_1645}.
We thus work in a complete, connected, $n$-dimensional Riemannian manifold $M$, with  $n \geq 2$.
We assume that $\iota: X \rightarrow M$ is an isometric embedding of the $\delta$-net
$ X \subseteq \RR^n$, and assume for convenience that $0 \in X$. 
For ease of reading, and with a slight abuse of notation, we identify between a point
$a \in X \subseteq \RR^n$ and its image $\iota(a) \in M$. 
Thus we think of $X$ as a subset of $M$, and the assumption that $\iota$ is an isometric embedding translates to
\begin{equation}  
	d(x,y) = |x-y| \qquad \qquad \qquad \text{for all} \ x, y \in X, \label{eq_600}
\end{equation}
where $d$ is the Riemannian distance function in $M$.
Note that for  $p \in X \subseteq M$, we may speak of the Euclidean norm $|p| = \sqrt{\sum_i p_i^2} = d(0,p)$ and of the scalar product $\langle p, v \rangle = \sum_i p_i v_i$ for $v \in \RR^n$.

\medskip
We would like to define a continuous map on a triangulation of the Euclidean space. 
It will therefore be more convenient for us to work with a certain lattice in $\RR^{n}$ rather than with the net $X$.
For $\eps > 0$ we write 
$$
L_{\eps} = \eps \cdot \ZZ^{n} \subseteq \RR^{n}
$$ 
for an $\eps$-scale of the lattice $\ZZ^{n}$.
For a point $p \in L_{\eps}$ we write
$$
Q_{\eps}(p) = p + [0,\eps]^{n} \subseteq \RR^{n}
$$
for a cube of side-length $\eps$, and $\text{ver}(Q_{\eps}(p))$ for its set of vertices.
We say that $\nu : \RR^{n} \to X$ is a 
\textit{net-rounding} map if for any $x \in \RR^{n}$ we have $\delta_{\nu}(x) := |x-\nu(x)| < \delta$.
Since $X$ is a $\delta$-net in $\RR^{n}$ we know that such a map exists.

\begin{lemma}
Suppose that for any $\eps > 0$ and any net-rounding map $\nu : \RR^{n} \to X$ there exists a surjective map 
$ \Phi_{\eps, \nu} : \RR^{n} \to M$ such that for any $p \in L_{\eps}$ and $x \in Q_{\eps}(p)$ we have
\begin{equation}
d(\Phi_{\eps, \nu}(x), \nu(p)) \leq n \cdot \left( \eps + 2 \max \left\{ \delta_{\nu}(q) : q \in \text{ver}(Q_{\eps}(p)) \right\}\right).
\label{eq_1838}
\end{equation}
Then $X$ is a $(2\delta n)$-net with respect to the Riemannian distance in $M$.
\label{lem_1251}
\end{lemma}

\begin{proof}
Write $B_{\RR^{n}}(x,r)$ for the open Euclidean ball of radius $r > 0$ around $x \in \RR^{n}$, and $\overline{B}_{\RR^{n}}(x,r)$ for its closure.
Let $y \in M$, and let $\nu: \RR^{n} \to X$ be some arbitrary net-rounding map.
By the assumptions of the lemma, there exists a surjective map $\Phi_{1,\nu} : \RR^{n} \to M$ for which (\ref{eq_1838}) holds true.
Let $x_{1} \in \RR^{n}$ and $p_{1} \in L_{1}$ be such that $\Phi_{1,\nu}(x_{1}) = y$ and $x_{1} \in Q_{1}(p_{1})$. Then 
\begin{equation}
d(y, \nu(p_{1})) = d(\Phi_{1,\nu}(x_{1}),\nu(p_{1})) < n \cdot (1 + 2\delta).
\label{eq_1236}
\end{equation}
Set $R = 2n \left( 1 + 2\delta \right) + 2\delta + \sqrt{n}$.
Since the Euclidean distance from the set $X$ is a continuous function, there exists $\tilde{\delta} < \delta$ such that
$ d_{\RR^{n}}(z,X) < \tilde{\delta}$ for any $z \in \overline{B}_{\RR^{n}}(p_{1},R)$.
We may therefore define a net-rounding map $\tilde{\nu} : \RR^{n} \to X$ such that 
\begin{equation}
\delta_{\tilde{\nu}}(z) = |\tilde{\nu}(z) - z| < \tilde{\delta} 
\qquad \text{for any} \quad 
z \in \overline{B}_{\RR^{n}}(p_{1},R).
\label{eq_1245}
\end{equation}
By the assumptions of the lemma, for any $0 < \eps < 1$ there exists a surjective map  $\Phi_{\eps, \tilde{\nu}} : \RR^{n} \to M$ for which (\ref{eq_1838}) holds true.
Thus there exist $x_{\eps} \in \RR^{n}$ and $p_{\eps} \in L_{\eps}$ such that $\Phi_{\eps, \tilde{\nu}}(x_{\eps}) = y$, $x_{\eps} \in Q_{\eps}(p_{\eps})$ and
\begin{equation}
d(y,\tilde{\nu}(p_{\eps})) = 
d(\Phi_{\eps,  \tilde{\nu}}(x_{\eps}),\tilde{\nu}(p_{\eps})) < 
n \cdot (1 + 2\delta).
\label{eq_1237}
\end{equation}
By the triangle inequality together with (\ref{eq_1236}) and (\ref{eq_1237}) we obtain that
\begin{equation}
|p_{\eps} - p_{1}| < d(\tilde{\nu}(p_{\eps}),\nu(p_{1})) + 2\delta \leq d(\tilde{\nu}(p_{\eps}),y) + d(y,\nu(p_{1})) + 2\delta <  2n \left( 1 + 2\delta \right) + 2\delta.
\label{eq_1240}
\end{equation}
Moreover, for any $q \in \text{ver}(Q_{\eps}(p_{\eps}))$ we have that $|q - p_{\eps}| \leq \eps \sqrt{n} < \sqrt{n}$. It therefore follows from (\ref{eq_1240}) and the triangle inequality that
\begin{equation}
\text{ver}(Q_{\eps}(p_{\eps})) \subseteq B_{\RR^{n}}(p_{1}, R).
\label{eq_1243}
\end{equation}
Then for $\eps_{0} = \min \{ 1/2, \delta - \tilde{\delta} \}$ we obtain from (\ref{eq_1838}), (\ref{eq_1245}) and (\ref{eq_1243}) that 
$$
d(y, \tilde{\nu}(p_{\eps_{0}})) = 
d(\Phi_{\eps_{0}, \tilde{\nu}}(x_{\eps_{0}}),\tilde{\nu}(p_{\eps_{0}})) <
n \cdot (\eps_{0} + 2\tilde{\delta})\leq 
n \cdot (\delta + \tilde{\delta}) < 2\delta n. 
$$
Therefore $d(y,X) < 2\delta n$, and the proof is completed.
\end{proof}

From now until the end of the section, we fix $\eps > 0$ and a net-rounding map $\nu : \RR^{n} \to X$.
According to Lemma \ref{lem_1251}, in order to prove Theorem \ref{thm_1730} it suffices to construct a surjective map $\Phi_{\eps, \nu} : \RR^{n} \to M$ that satisfies (\ref{eq_1838}) .
We omit the $\eps$ and $\nu$ subscripts from the notation described above.
Hence the lattice $\eps \cdot \ZZ^{n}$ will be denoted by $L$, the cube of side-length $\eps$ by $Q(p)$, and for $x \in \RR^{n}$ we will write $\delta(x) = |x-\nu(x)|$.
For convenience we will use the notation $\tilde{x} = \nu(x)$ for a point $x \in \RR^{n}$. 
Since $\nu$ is a net-rounding map, by the triangle inequality we have that
\begin{equation}
\left| d(\tilde{x}_{1}, \tilde{x}_{2}) - |x_{1} - x_{2}| \right| \leq \delta(x_{1}) + \delta(x_{2})< 2\delta
\qquad \text{for any } x_{1},x_{2} \in \RR^{n}.
\label{eq_1651}
\end{equation}

\subsection{Continuous image of a simplex}
\label{subsec_simplex}

The primary building block of our construction of $\Phi$ is a continuous image of a $k$-simplex in $\RR^{n}$, for $0 \leq k \leq n$.
The points $\{ p_{1},p_{2},...,p_{k}\} \subseteq \RR^{n}$ are said to be affinely independent if the vectors $\{p_{2}-p_{1},p_{3}-p_{1},...,p_{k}-p_{1}\}$ are linearly independent. 
A $k$-simplex in $\RR^{n}$ is the convex hull of $k+1$ affinely independent points.
We denote the convex hull of the points $p_{1},...,p_{k} \in \RR^{n}$ by $\text{conv}(p_{1},....,p_{k}) \subseteq \RR^{n}$, and we refer to a point $\sum_{j=1}^{k}\lambda_{j}p_{j} \in \text{conv}(p_{1},...,p_{k})$ by its barycentric coordinates $(\lambda_{1},\lambda_{2},...,\lambda_{k})$.
The maps we define are parametrized by \textit{ordered} sequences of vertices.
Given points $p_{1},...,p_{k} \in \RR^{n}$, a subset $J \subseteq \{1,...,k\}$ naturally corresponds to an ordered subsequence of these points: if $J = \{j_{1},...,j_{m}\}$ where $1 \leq j_{1} < j_{2} < ... < j_{m} \leq k$, then we denote
$$
J(p_{1},...,p_{k}) = (p_{j_{1}},p_{j_{2}},...,p_{j_{m}}).
$$
Let us construct a ``simplex mapping'' $\Delta$ that for any $0 \leq k \leq n$ and $k+1$ affinely independent points $u_{0},u_{1},...,u_{k} \in \RR^{n}$ assigns a continuous map
$$
\Delta[u_{0},u_{1},...,u_{k}] : \text{conv}(u_{0},u_{1},...,u_{k}) \to M,
$$
such that the following hold:
\begin{enumerate}
\item [(i)] For $k = 0$ we have $\Delta[u_{0}] = \nu$, i.e. $\Delta[u_{0}](u_{0}) = \tilde{u}_{0}$.

\item[(ii)] The mapping $\Delta$ is consistent with the simplicial complex structure, i.e. for any $1 \leq k \leq n$ and affinely independent $u_{0},u_{1},...,u_{k} \in \RR^{n}$ we have that
$$
\Delta[u_{0},u_{1},...,u_{k}] \big|_{\text{conv}(J(u_{0},u_{1},...,u_{k}))} = 
\Delta[J(u_{0},u_{1},...,u_{k})]
\qquad \text{for any } J \subseteq \{0,1,...,k\}.
$$
\item [(iii)] For any $1 \leq k \leq n$ and affinely independent $u_{0},u_{1},...,u_{k} \in \RR^{n}$ we have
$$
d (y, \tilde{u}_{0}) \leq 
\sum_{i=0}^{k-1} d (\tilde{u}_{i},\tilde{u}_{i+1}) 
\qquad \text{for any } y \in \Im \left( \Delta[u_{0},u_{1},...,u_{k}] \right),
$$
\end{enumerate}
where $\Im(f)$ stands for the image of the map $f$.
The definition of $\Delta$ is done by induction on $k$.
For a single point $u_{0} \in \RR^{n}$ we define $\Delta[u_{0}]$ as required by (i). 
Suppose that $1 \leq m \leq n$ and let $u_{0},...,u_{m} \in \RR^{n}$ be affinely independent.
Write 
$$
\Delta_{0}^{m-1} = \Delta[u_{1},...,u_{m}]: \text{conv}(u_{1},...,u_{m}) \to M
$$
for the continuous map given by the induction hypothesis.
For any point given in barycentric coordinates by $(\theta_{1},...,\theta_{m}) \in \text{conv}(u_{1},...,u_{m})$ there exists a unique minimizing geodesic segment of constant speed $\gamma_{(\theta_{1},...,\theta_{m})}: [0,1] \to M$ such that 
$$
\gamma_{(\theta_{1},...,\theta_{m})}(0) = \tilde{u}_{0}
\qquad \text{and} \qquad \gamma_{(\theta_{1},...,\theta_{m})}(1) = \Delta_{0}^{m-1}(\theta_{1},...,\theta_{m}).
$$
In case $\tilde{u}_{0} = \Delta_{0}^{m-1}(\theta_{1},...,\theta_{m})$, we simply let $\gamma_{(\theta_{1},...,\theta_{m})}$ be the constant path.
Define the map $\Delta^{m} = \Delta[u_{0},u_{1},...,u_{m}] : \text{conv}(u_{0},u_{1},...,u_{m}) \to M$ by
$$
\Delta^{m}(\lambda_{0},\lambda_{1},...,\lambda_{m}) =
\begin{cases}
\tilde{u}_{0} , & \text{if } \lambda_{0} = 1,  \\
\gamma_{\left( \frac{\lambda_{1}}{1-\lambda_{0}},\frac{\lambda_{2}}{1-\lambda_{0}},..., \frac{\lambda_{m}}{1-\lambda_{0}} \right)}(1-\lambda_{0}), & \text{if } 0 \leq \lambda_{0} < 1.
\end{cases}
$$

\begin{figure}
\begin{center} \includegraphics[width=4in]{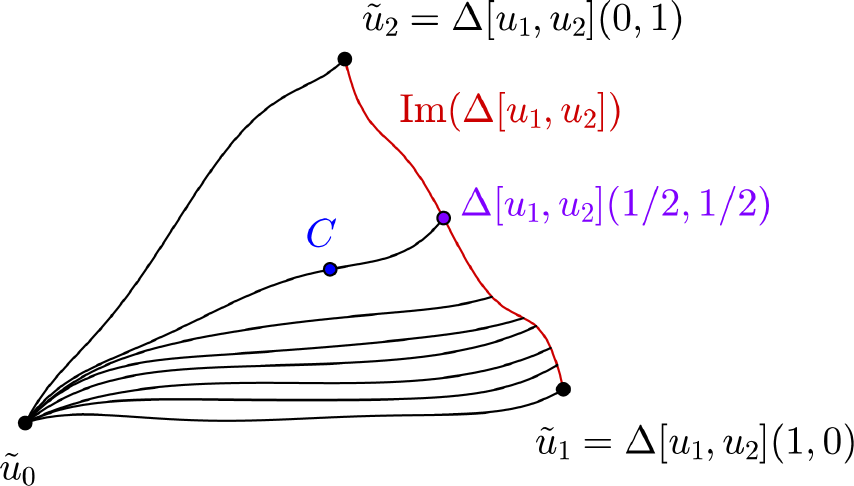} \end{center}
\caption{An illustration of the recursive construction of the map $\Delta[u_{0},u_{1},u_{2}]$. The line in $\RR^{n}$ connecting $u_{1}$ and $u_{2}$ is mapped to the minimizing geodesic connecting the points $\tilde{u}_{1}$ and $\tilde{u}_{2}$ in $M$, illustrated by the \textcolor{red}{red} segment.
Then the interior of the triangle in $\RR^{n}$ with vertices $u_{0}$, $u_{1}$ and $u_{2}$ is mapped to the union of geodesic segments connecting $\tilde{u}_{0} \in M$ to the above-mentioned segment.
As an example, the point $C$ is drawn as the image of the triangle's centroid, that is $C = \Delta[u_{0},u_{1},u_{2}](1/3,1/3,1/3)$.}
\end{figure}

\noindent The continuity of $\Delta^{m}$ follows from the continuity of $\Delta_{0}^{m-1}$, which is assumed by the induction hypothesis, and the fact that the exponential map at $\tilde{u}_{0}$ is a diffeomorphism, thanks to \cite[Corollary 4.1]{EK}. Let us show that property (iii) holds true.
In our construction, any point of $\text{Im}(\Delta^{m})$ lies on a minimizing geodesic segment connecting $\tilde{u}_{0}$ and a point of $\Im(\Delta_{0}^{m-1})$. Together with the induction hypothesis and the triangle inequality, we therefore obtain that
\begin{align*}
\sup \{d(\tilde{u}_{0},y) : y \in \Im(\Delta^{m})\} &=
\sup \{d(\tilde{u}_{0},y) : y \in \Im(\Delta_{0}^{m-1}) \} 
\\ & \leq 
d(\tilde{u}_{0},\tilde{u}_{1}) + \sup \{ d(\tilde{u}_{1},y) : y \in \Im(\Delta_{0}^{m-1})\} \leq 
\sum_{i=0}^{m-1} d(\tilde{u}_{i},\tilde{u}_{i+1}),
\end{align*}
so that property (iii) holds true.
By the induction hypothesis, for the proof of property (ii) it suffices to show that
$$
\Delta^{m} \big|_{\text{conv}(u_{0},...,\hat{u}_{j},...,u_{m})} = 
\Delta[u_{0},...,\hat{u}_{j},...,u_{m}] =: \Delta^{m-1}_{j}
\qquad \text{for any } \; 0 \leq j \leq m, 
$$
where $\hat{u}_{j}$ stands for the omission of the element with index $j$ from the expression.
For $j=0$ this is clear, since for any $(\lambda_{1},...,\lambda_{m}) \in \text{conv}(u_{1},...,u_{m})$ we have
$$
\Delta^{m}(0,\lambda_{1},...,\lambda_{m}) =  \gamma_{(\lambda_{1},...,\lambda_{m})}(1) = \Delta_{0}^{m-1}(\lambda_{1},...,\lambda_{m}).
$$
The proofs for all other indices are symmetric, hence it suffices to show that
\begin{equation}
\Delta^{m} \big|_{\text{conv}(u_{0},u_{2},...,u_{m}) } = \Delta_{1}^{m-1}.
\label{eq_1635}
\end{equation}
For any $(\theta_{2},...,\theta_{m}) \in \text{conv}(u_{2},...,u_{m})$ let $\eta_{(\theta_{2},...,\theta_{m})}: [0,1] \to M$ be the minimizing geodesic of constant speed with
$$
\eta_{(\theta_{2},...,\theta_{m})}(0) = \tilde{u}_{0}
\qquad \text{and} \qquad \eta_{(\theta_{2},...,\theta_{m})}(1) = \Delta[u_{2},...,u_{m}](\theta_{2},...,\theta_{m}).
$$ 
According to the induction hypothesis, for any $(\theta_{2},...,\theta_{m}) \in \text{conv}(u_{2},...,u_{m})$ we have that
$$
\gamma_{(0,\theta_{2},...,\theta_{m})}(1) = \Delta_{0}^{m-1}(0, \theta_{2},...,\theta_{m}) = \Delta[u_{2},...,u_{m}](\theta_{2},...,\theta_{m}) = \eta_{(\theta_{2},...,\theta_{m})}(1).
$$
Hence the geodesics $\gamma_{(0,\theta_{2},...,\theta_{m})}$ and $\eta_{(\theta_{2},...,\theta_{m})}$ coincide, as they are both defined to be the minimizing geodesic segment connecting the same endpoints.
Thus for any $(\lambda_{0},\lambda_{2},...,\lambda_{m}) \in \text{conv}(u_{0},u_{2},...,u_{m})$ we have that
$$
\Delta^{m}(\lambda_{0},0,\lambda_{2},...,\lambda_{m}) = 
\begin{cases}
\tilde{u}_{0} , & \text{if } \lambda_{0} = 1,  \\
\eta_{\left( \frac{\lambda_{2}}{1-\lambda_{0}},..., \frac{\lambda_{m}}{1-\lambda_{0}} \right)}(1-\lambda_{0}), & \text{if } 0 \leq \lambda_{0} < 1
\end{cases}
= \Delta_{1}^{m-1}(\lambda_{0},\lambda_{2},...,\lambda_{m}).
$$
Therefore (\ref{eq_1635}) holds true, and property (ii) is valid as well.

\subsection{Gluing the simplex maps} \label{subsec_gluing}

In order to obtain the continuous map $\Phi$ defined on the entire $\RR^{n}$, we glue together $\Delta$-maps. 
According to property (ii), the gluing is possible as long as the order of the vertices is consistent.
We consider the lexicographic order on $\RR^{n}$: for two different points 
$x = (x_{1},...,x_{n})\in \RR^{n}$ and 
$y = (y_{1},...,y_{n})\in \RR^{n}$ we say that $x \prec y$ if $x_{j_{0}} < y_{j_{0}}$ for $j_{0} = \min \{ 1 \leq  j \leq n : x_{j} \neq y_{j} \}$. 
For a simplex $\sigma \subseteq \RR^{n}$ of dimension $k$ we enumerate its vertices by $\text{ver}(\sigma) = \{v_{0},v_{1},...,v_{k}\}$ where $v_{0} \prec v_{1} \prec ... \prec v_{k}$, and denote the corresponding continuous $\Delta$-map by
$$
\Delta[\sigma] = \Delta[v_{0},v_{1},...,v_{k}] : \sigma \to M.
$$ 
If $\sigma_{1},\sigma_{2} \subseteq \RR^{n}$ are two simplices that share a common face $\tau$, then by property (ii) we have that
\begin{equation}
\Delta[\sigma_{1}] \big|_{\tau} = 
\Delta[\tau] = 
\Delta[\sigma_{2}] \big|_{\tau}.
\label{eq_1729}
\end{equation}
Recall that for a point $p \in L$ we defined the cube $Q(p) = p + [0,\eps]^{n} \subseteq \RR^{n}$ of side-length $\eps$, and denoted its set of vertices by $\text{ver}(Q(p))$.
Such a cube admits a triangulation which consists of $n!$ simplices.
Write $S_{n}$ for the group of permutations on $\{ 1,...,n \}$. 
For a point $p \in L$ and a permutation $\pi \in S_{n}$ we define the simplex
\begin{equation}
\sigma (p,\pi) = \text{conv} \left\{ p, \;
p + \eps e_{\pi(1)}, \;
p + \eps e_{\pi(1)} + \eps e_{\pi(2)},..., \; 
p + \sum_{l = 1}^{n} \eps e_{\pi(l)} 
\right\} \subseteq \RR^{n},
\label{eq_1642}
\end{equation}
where $e_{j}$ is the $j$-th element of the standard basis in $\RR^{n}$.
Then $Q(p) = \bigcup_{\pi \in S_{n}} \sigma(p, \pi)$. 
Moreover, by property (iii) together with (\ref{eq_1651}), for any $p \in L$ and $\pi \in S_{n}$ we have
\begin{align}
\sup \{ d (\tilde{p}, y) : y \in \Im(\Delta[\sigma(p,\pi)])\} &\leq 
\sum_{m = 1}^{n} d \left( 
\nu \left( p + \sum_{l=1}^{m-1} \eps e_{\pi(l)}\right) , 
\nu \left( p + \sum_{l=1}^{m} \eps e_{\pi(l)} \right) \right) 
\nonumber \\& \leq n \cdot \left( \eps + 2 \max \left\{ \delta(q) : q \in \text{ver}(Q(p)) \right\} \right).
\label{eq_1636}
\end{align}
Thus by gluing together all the maps $\{\Delta[\sigma(p, \pi)] : p \in L, \; \pi \in S_{n}\}$, which is possible thanks to (\ref{eq_1729}), we obtain a continuous map $\Phi : \RR^{n} \to M$ for which assumption (\ref{eq_1838}) of Lemma \ref{lem_1251} holds true. In order to show that the assumptions of Lemma \ref{lem_1251} are satisfied, and consequently prove Theorem \ref{thm_1730}, it therefore remains to show that $\Phi$ is surjective.

\subsection{Continuous additive quasi-isometries are surjective} \label{subsec_surjective}

We say that a map $\varphi: N \to M$ between two metric spaces $(N, d_{N})$ and $(M, d_{M})$ is a $(\lambda, R)$-quasi-isometry for $\lambda \geq 1$ and $R \geq 0$ if
\begin{equation}
\lambda^{-1} d_{N}(x_{1},x_{2}) - R \leq
d_{M}(\varphi(x_{1}),\varphi({x_{2}})) \leq
\lambda d_{N}(x_{1},x_{2}) + R
\qquad \text{for any } x_{1},x_{2} \in N.
\label{eq_1512}
\end{equation}
If $\lambda = 1$ we say that $\varphi$ is an additive quasi-isometry.

\begin{proposition}
Let $1 \leq \lambda < \sqrt{3}$ and $R \geq 0$. Let $\varphi: N \to M$ be a continuous $(\lambda, R)$-quasi-isometry between two complete, connected $n$-dimensional Riemannian manifolds, and suppose that there exists a point $o \in N$ such that:
\begin{enumerate}
\item[(i)] all geodesics in $N$ passing through $o$ are minimizing throughout, and
\item[(ii)] the exponential map $\exp_{\varphi(o)}: T_{\varphi(o)}M \to M$ is a diffeomorphism.
\end{enumerate} 
Then $\varphi$ is surjective.
\label{prop_1711}
\end{proposition}

\begin{proof}
Using assumption (ii), we may define the continuous radial projection $v_{o}: M \setminus \{\varphi(o)\} \to S_{\varphi(o)}M = \{ u \in T_{\varphi(0)} M \; ; \; \Vert u \Vert = 1\}$ by 
$$
v_{o}(y) = \frac{\exp_{\varphi(o)}^{-1}(y)}{\Vert \exp_{\varphi(o)}^{-1}(y) \Vert}.
$$
Since geodesic rays emanating from $\varphi(o)$ are minimizing, for any $y_{1},y_{2} \in M \setminus \{ \varphi(o)\}$ we have that
\begin{equation}
v_{o}(y_{1}) = v_{o}(y_{2}) 
\implies
|d_{M}(\varphi(o), y_{1}) - d_{M}(\varphi(o), y_{2})| = 
d_{M}(y_{1},y_{2}).
\label{eq_1543}
\end{equation}
For $r \geq 0$ define $\varphi_{r} : S_{o}N \to M$ by $\varphi_{r}(u) = \varphi(\gamma_{u}(r))$, where $\gamma_{u}: [0, \infty) \to N$ is the minimizing geodesic ray emanating from $o$ in direction $u$. 
According to (\ref{eq_1512}) we have that
\begin{equation}
\lambda^{-1}r - R \leq d_{M}(\varphi(o), \varphi_{r}(u)) \leq \lambda r + R
\qquad \text{for any } r \geq 0 \text{ and } u \in S_{o}N.
\label{eq_1524}
\end{equation}
Hence for any $r > \lambda R$ we have $\varphi_{r}(S_{o}N) \subseteq M \setminus \{\varphi(o)\}$, and therefore $v_{o} \circ \varphi_{r} : S_{o}N \to S_{\varphi(o)}M$ is well-defined. 
Moreover, using (\ref{eq_1524}) we see that for any $r > 0$ and $u,v \in S_{o}N$,
\begin{equation}
\left| 
d_{M}(\varphi(o), \varphi_{r}(u)) - 
d_{M}(\varphi(o), \varphi_{r}(v))
\right| \leq 
( \lambda - \lambda^{-1} ) r + 2R.
\label{eq_1537}
\end{equation}
Combining (\ref{eq_1512}) together with assumption (i), for any $r > 0$ and $u \in S_{o}M$ we have that
\begin{equation}
d_{M}(\varphi_{r}(u), \varphi_{r}(-u)) \geq 
\lambda^{-1}d_{N}(\gamma_{u}(r), \gamma_{u}(-r)) - R =
2\lambda^{-1}r - R.
\label{eq_1531}
\end{equation}
Therefore by (\ref{eq_1537}) and (\ref{eq_1531}), for any $r > R_{0} := 3\lambda R / (3- \lambda^{2})$ and $u \in S_{o}N$ we have
$$
\left| 
d_{M}(\varphi(o), \varphi_{r}(u)) - 
d_{M}(\varphi(o), \varphi_{r}(-u))
\right| < d_{M}(\varphi_{r}(u), \varphi_{r}(-u)).
$$
Hence by (\ref{eq_1543}) we obtain that
\begin{equation} 
v_{o}(\varphi_{r}(u)) \neq v_{o}(\varphi_{r}(-u))
\qquad \text{for any } r > R_{0} \text{ and } u \in S_{o}N.
\label{eq_1545}
\end{equation}
	
\medskip
Now suppose that $\varphi$ is not surjective, i.e. $\varphi(N) \subseteq M \setminus \{y\}$ for some $y \in M$. Let $\gamma$ be a minimizing geodesic connecting $\varphi(o)$ and $y$, and let $T := \{ z \in M : d_{M}(z,\gamma) < 1 \}$ be a neighborhood of $\gamma$. Since $T$ is open and connected, there exists a diffeomorphism $F : M \to M$ such that $F \big|_{M \setminus T}$ is the identity and $F(y) = \varphi(o)$. Then $(F \circ \varphi)(N) \subseteq M \setminus \{\varphi(o)\}$	and for any $r > 0$ the map $H_{r} : S_{o}N \times [0,1] \to M \setminus \{ \varphi(o)\}$ given by $H_{r}(u,t) = F(\varphi_{tr}(u))$ is well-defined and continuous.
For any $z \in T$ we have that $d_{M}(z, \varphi(o)) < 1 + d_{M}(y, \varphi(o))$.
Thus from (\ref{eq_1512}) it follows that
$$
\varphi_{r}(S_{o}N) \subseteq M \setminus T
\qquad \text{for any } r > R_{1} := \lambda  \left( d_{M}(\varphi(o),y) + 1 + R \right).
\label{eq_1559}
$$
Hence for any $r > R_{1}$ we have that $F \circ \varphi_{r} = \varphi_{r}$ and we see that $H_{r}$ is in fact a null-homotopy of $\varphi_{r}$. 
Therefore $v_{o} \circ H_{r} : S_{o}N \times [0,1] \to S_{\varphi_{o}}M$ is a null-homotopy of $v_{o} \circ \varphi_{r}$.
However, according to the Borsuk-Ulam theorem, we see that (\ref{eq_1545}) implies that such a null-homotopy does not exist whenever $r > R_{0}$.
\end{proof}

Observe that the continuous map $\Phi: \RR^{n} \to M$ constructed in the previous section is an additive quasi-isometry. 
Indeed, since it satisfies assumption (\ref{eq_1838}) of Lemma \ref{lem_1251}, there exists $C > 0$ such that $d(\Phi(x), \tilde{p}) \leq C$ for any $p \in L$ and $x \in Q(p)$. Hence for $x_{1},x_{2} \in \RR^{n}$ we let $p_{1}, p_{2} \in L$ be such that $x_{i} \in Q(p_{i})$ for $i \in \{1,2\}$, and we have that
$ |x_{i} - \tilde{p}_{i}| \leq \eps \sqrt{n} + \delta $
and
$ d(\Phi(x_{i}), \tilde{p}_{i}) \leq C$. It therefore follows from the triangle inequality that 
$$
\left| d(\Phi(x_{1}), \Phi(x_{2})) - | x_{1} - x_{2} | \right| \leq
2 \left( C + \eps \sqrt{n} + \delta \right),
$$
so that $\Phi$ is an additive quasi-isometry.
Moreover, according to \cite[Corollary 4.1]{EK} we have that assumptions (i) and (ii) from Proposition \ref{prop_1711} hold for any point of $X$. Hence we conclude that $\Phi$ is surjective.
The proof of Theorem \ref{thm_1730} is therefore completed according to Lemma \ref{lem_1251}.

\section{The ideal boundary}
\label{sec_2}

Let us move on to the proof of Theorem \ref{thm_1645}.
We say that a geodesic $\gamma: \RR \to M$ is a transport line of a $1$-Lipschitz function $f: M \to \RR$ if 
$$
f(\gamma(t)) - f(\gamma(s)) = t-s
\qquad \text{for all } s,t \in \RR.
$$ 
Thus the function $f$ grows with unit speed along a transport line, and any transport line is a distance minimizing geodesic. 
Moreover, the function $f$ is differentiable at $\gamma(t)$ for any $t \in \RR$, and we have
\begin{equation}
\nabla f(\gamma(t)) = \dot{\gamma}(t).
\label{eq_1740}
\end{equation}
We say that a $1$-Lipschitz function $f: M \rightarrow \RR$ induces a foliation by transport lines, or in short foliates, if
for any $x \in M$ there exists a transport line of $f$ that contains $x$.
In this case $M$ is the disjoint union of the transport lines of $f$, and the function $f$ is differentiable everywhere in $M$.
The following lemma states that the gradient of a continuous variation of $1$-Lipschitz foliating functions varies continuously as well.
A proof for the lemma is given in \cite[Lemma 2.5]{EK}.

\begin{lemma} Let $V$ be a metric space, and assume that with any $v \in V$ we associate a $1$-Lipschitz function $f_{v}: M \rightarrow \RR$.
	Suppose that $f_{v}$ foliates for any $v \in V$, and that $f_{v}(x)$ varies continuously with $v \in V$ for any fixed $x \in M$. Then the map
	$$ 
	(x, v) \mapsto \nabla f_{v}(x) 
	$$
	is continuous in $M \times V$.
	\label{lem_1835}
\end{lemma}

\medskip
In \cite{EK} we introduce the following notion, in the spirit of the ``ideal boundary'' of a Hadamard manifold. See \cite{BG} for information about the ideal boundary.
For $v \in S^{n-1}$ we write
$$
\partial_{v}M = 
\{ B : M \to \RR  \, ; \, B \text{ is } 1 \text{-Lipschitz with } B(p) = \langle p,v \rangle \text{ for all } p \in X\}.
$$
The set $\partial_{v}M$ is non-empty for any $v \in S^{n-1}$ according to \cite[Lemma 3.4]{EK}, and clearly
\begin{equation}  
	\partial_{-v} M = -\partial_v M := \{ -B \, ; \, B \in \partial_v M \}. 
	\label{eq_1929}
\end{equation}
For $v \in S^{n-1}$ we can therefore define
$$
B_{v}(x) = \inf_{B \in \partial_{v}M} B(x),
\qquad
B^{v}(x) = \sup_{B \in \partial_{v}M} B(x) = -B_{-v}(x),
\qquad \text{and} \qquad
f_{v} = B^{v} - B_{v}.
$$
The functions $B_{v}$ and $B^{v}$ both foliate and belong to $\partial_{v}M$, as shown in \cite[Proposition 4.3]{EK}. The difference $f_{v}$ is easily seen to be non-negative.
A sequence of points $p_m \in \RR^n \ \ (m=1,2,\ldots)$ is said to be drifting in the direction of $v \in S^{n-1}$, and we write $ p_m \rightsquigarrow v $, if
$$
|p_m| \xrightarrow{m \to \infty} \infty \qquad \text{and} \qquad \frac{p_m}{|p_m|} \xrightarrow{m \to \infty} v. 
$$
We say that $p_{m}  \rightsquigarrow v$ narrowly if 
$$
|p_m| \xrightarrow{m \rightarrow \infty} \infty  \qquad \text{and} \qquad |p_m| - \langle p_m, v \rangle \xrightarrow{m \rightarrow \infty}  0. 
$$
The assumption that $X$ is a net in $\RR^{n}$ implies that for any $v \in S^{n-1}$ there exists a sequence $(p_{m})_{m \geq 1}$ in $X$ such that $p_{m} \rightsquigarrow v$ narrowly, see \cite[Lemma 3.3]{EK} for a proof. 
Given $p \in X$ we write $d_{p}:M\to\RR$ for the function
$$ 
d_p(x) = d(p, 0) - d(p, x),
$$
which is a $1$-Lipschitz function that vanishes at $0 \in X \subseteq M$.
If $p_{m} \rightsquigarrow v$ narrowly, then $d_{p_{m}} \longrightarrow B_{v}$ locally-uniformly as $m \rightarrow \infty$, as explained in the proof of \cite[Proposition 4.3]{EK}.
The following lemma states that Theorem \ref{thm_1645} follows from the fact that $B_{v} \equiv B^{v}$ for all $v \in S^{n-1}$. A proof based on \cite[Lemma 6.1]{EK} and \cite[Corollary 6.2]{EK} is given below.

\begin{lemma} Suppose that for any $v \in S^{n-1}$ we have $B_{v} \equiv B^{v}$, i.e that $\partial_{v}M$ is a singleton.
Then for any $x \in M$, the map  $S^{n-1} \ni v \mapsto \nabla B_v(x) \in S_x M$ is continuous and onto. Therefore all geodesics in $M$ are minimizing, and there are no conjugate points in $M$.
	\label{lem_1841}
\end{lemma}

\begin{proof}
	Let us show that $v \mapsto B_{v}(x)$ is continuous in $v \in S^{n-1}$ for any fixed $x \in M$.
	Recall that the $1$-Lipschitz function $B_v$ vanishes at $0$ for any $v \in S^{n-1}$. Let $v_m \longrightarrow v$ be a sequence in $S^{n-1}$.
	By the Arzela-Ascoli theorem, we may pass to a subsequence and assume that $(B_{v_m} )_{m \geq 1}$ converges locally-uniformly to some $1$-Lipschitz function $B$.
	For any $p\in X$ we have
	$$
	B(p)=\lim_{m \rightarrow \infty} B_{v_m}(p)=\lim_{m \rightarrow \infty}\langle p, v_m\rangle=\langle p, v\rangle,
	$$
	and hence $B \in \partial_v M$. 
	Using the assumption that $\partial_{v}M$ is a singleton we have that $B \equiv B_{v}$ and the map $v \mapsto B_v(x)$
	is continuous in $v \in S^{n-1}$.
	Since the function $B_{v}$ foliates, Lemma \ref{lem_1835} implies the continuity of the map $S^{n-1} \ni v \to \nabla B_{v}(x) \in S_{x}M$.
	Using the fact that  $-B_v \in \partial_{-v} M$ and the assumption that $\partial_{-v}M$ is a singleton, we obtain that $-B_{v} \equiv B_{-v}$ for any $v \in S^{n-1}$.
	Thus $v \mapsto \nabla B_v(x)$ is a continuous, odd map from $S^{n-1}$ to $S_{x} M$, hence its Brouwer degree is odd
	and the map is onto.
	
	\medskip
	Therefore for any $x \in M$ and a complete geodesic $\gamma$ with $\gamma(0) = x$, there exists $v \in S^{n-1}$ such that $\dot{\gamma}(0) = \nabla B_v(x)$.
	Since the function $B_{v}$ foliates, the geodesic $\gamma$ is a transport line of $B_{v}$, and hence it is a minimizing geodesic.
	Since a complete, minimizing geodesic cannot contain a pair of conjugate points, there are no conjugate points in $M$. 
\end{proof}

In view of Lemma \ref{lem_1841}, it is therefore our goal to prove the following proposition, for which we dedicate the following sections.

\begin{proposition}
	For any $v \in S^{n-1}$ we have $f_{v} \equiv B^{v} - B_{v} \equiv 0$ and hence $\partial_{v}M$ is a singleton.
	\label{prop_1530}
\end{proposition}

\section{The deficit in the triangle inequality}
\label{sec_3}

Recall that a map $\nu: \RR^{n} \to X$ is net-rounding if for any $x \in \RR^{n}$ we have $|x - \nu(x)| < \delta$.
Let us fix some net-rounding map $\nu$, whose existence is assured by the fact that $X$ is a $\delta$-net in $\RR^{n}$, and again use the notation $\tilde{x} = \nu(x)$ for a point $x \in \RR^{n}$.
From now on and until the end of this section we also fix $v \in S^{n-1}$. 
For $\xi \in \RR^{n}$ and $t \in \RR$ define
$$
v(\xi,t) = \xi + tv \in \RR^{n}.
$$
We think of $v(\xi,t)$ as shifting the point $\xi$ by time $t$ in direction $v$. According to the above-mentioned notation, we denote the net-rounding of such a point by $\tilde{v}(\xi, t) = \nu(v(\xi ,t))$.
For $x \in M$, $\xi \in \RR^{n}$ and $t > 0$ define
$$
\delta(x,\xi,t) = 
d(x, \tilde{v}(\xi, t)) +
d(x, \tilde{v}(\xi, -t)) - 
d(\tilde{v}(\xi, t), \tilde{v}(\xi, -t)).
$$
The quantity $\delta(x,\xi, t)$ is the deficit in the triangle inequality with respect to the geodesic triangle in $M$ whose vertices are $\tilde{v}(\xi, t)$, $\tilde{v}(\xi, -t)$ and $x$.
Recall that for $x \in M$ we defined $f_{v}(x) = B^{v}(x) - B_{v}(x)$  which is non-negative. The function $f_{v}$ is in fact the limit of the deficit in the triangle inequality, as expressed in the following lemma.

\begin{lemma}
	For any $x \in M$ and $\xi \in \RR^{n}$	we have $\delta(x, \xi, t) \longrightarrow f_{v}(x)$ as $t \to \infty$.
	\label{lem_1930}
\end{lemma}

\begin{proof}
	Let $x \in M$ and $\xi \in \RR^{n}$. For $t > 0$ abbreviate  $q^{+}_{t} = \tilde{v}(\xi,t)$ and $q^{-}_{t} = \tilde{v}(\xi, -t)$.
Observe that $q^{+}_{t} \rightsquigarrow v$ narrowly and $q^{-}_{t} \rightsquigarrow -v$ narrowly as $t \to \infty$.
Hence we have that $d_{q^{+}_{t}}(x) \longrightarrow B_{v}(x)$ and $d_{q^{-}_{t}}(x) \longrightarrow B_{-v}(x)$ as $t \to \infty$.
Moreover, the deficit in the triangle inequality for the Euclidean triangle with vertices $q^{+}_{t}$, $q^{-}_{t}$ and the origin satisfies
$$
|q^{+}_{t}| + |q^{-}_{t}| - |q^{+}_{t} - q^{-}_{t}| \xrightarrow{t \to \infty} 0.
$$
Therefore we obtain that
	\begin{align*}
		\delta(x,\xi,t) &= d(x, q^{+}_{t}) + d(x, q^{-}_{t}) - d(q^{+}_{t}, q^{-}_{t}) 
		\\ &= 
		|q^{+}_{t}|- d_{q^{+}_{t}}(x) + |q^{-}_{t}|- d_{q^{-}_{t}}(x) - |q^{+}_{t} - q^{-}_{t}| \xrightarrow{t \to \infty} - B_{v}(x) - B_{-v}(x) = f_{v}(x),
	\end{align*}
	and the proof is completed.
\end{proof}

The following lemma is a rather simple implication of the triangle inequality that we will use later on in our construction.

\begin{lemma}
	Let $x \in M$, $\xi \in \RR^{n}$ and $t > 0$. Then for any $y$ lying on the geodesic segment connecting $x$ and $\tilde{v}(\xi,t)$ we have $\delta(y, \xi, t) \leq \delta(x, \xi, t)$.
	\label{lem_1522}
\end{lemma}

\begin{figure}
	\begin{center} \includegraphics[width=4.5in]{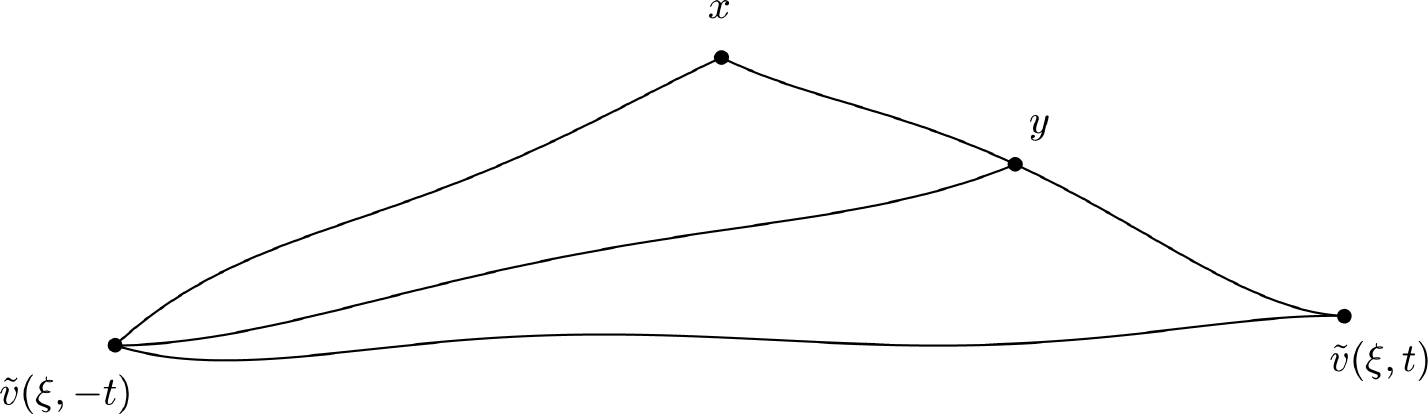} \end{center}
	\caption{An illustration of the geodesic triangles from Lemma \ref{lem_1522}.}
\end{figure}

\begin{proof}
	Since $y$ lies on the geodesic segment connecting $x$ and $\tilde{v}(\xi,t)$ we have that
	$$
	d(x, \tilde{v}(\xi, t)) = d(x,y) + d(y, \tilde{v}(\xi,t)).
	$$
	Therefore by the triangle inequality we obtain that
	\begin{align*}
		\delta(y, \xi, t) &= 
		d(y, \tilde{v}(\xi, -t)) +
		d(y, \tilde{v}(\xi,t)) - 
		d(\tilde{v}(\xi, t), \tilde{v}(\xi, -t))
		\\ &\leq 
		d(y,x) + d(x, \tilde{v}(\xi, -t)) +
		d(y, \tilde{v}(\xi,t)) - 
		d(\tilde{v}(\xi, t), \tilde{v}(\xi, -t)) = \delta(x, \xi, t),
	\end{align*}
	completing the proof.
\end{proof}

Let $\xi_{0} \in \RR^{n}$ and suppose that $\xi^{+}, \xi^{-} \in \RR^{n}$ are such that $\max \{ |\xi^{+} - \xi_{0}|, |\xi^{-} - \xi_{0}|\} \leq \sqrt{n}$. For $t,s > 0$ consider the Euclidean triangle in $\RR^{n}$ whose vertices are $\tilde{v}(\xi^{+},t)$,  $\tilde{v}(\xi^{-},-s)$ and $\tilde{\xi}_{0}$. As $t,s \rightarrow \infty$, the deficit in the triangle inequality tends to zero. This assertion clealy remains valid if all three points are shifted in direction $v$ by some fixed time $t_{0} \in \RR$.
Moreover, the convergence of this deficit in the triangle inequality is uniform in $\xi_{0}$ and in $t_{0}$.
In other words, for any $\eps > 0$ there exists $T = T(n, \delta, \eps) > 0$ large enough so that for any $\xi^{+},\xi^{-} \in \overline{B}_{\RR^{n}}(\xi_{0}, \sqrt{n})$, $t_{0} \in \RR$ and $t,s \geq T$ we have
\begin{equation}
	|\tilde{v}(\xi_{0}, t_{0}) - \tilde{v}(\xi^{+}, t_{0} + t)| + 
	|\tilde{v}(\xi_{0}, t_{0}) - \tilde{v}(\xi^{-}, t_{0} - s)| -
	|\tilde{v}(\xi^{+}, t_{0} + t) - \tilde{v}(\xi^{-}, t_{0} - s)| < \eps.
	\label{eq_2136}
\end{equation}
Since we fixed the dimension of the manifold $n \geq 2$ and the net constant $\delta > 0$, in the following we consider $T$ only as a function of $\eps > 0$.
Loosely speaking, the following lemma states that shifting vertices in opposite directions by some time larger than $T$ may only cause a small increase of at most $2\eps$ in the triangle inequality deficit.
\begin{lemma}
	Let $\eps > 0$ and $T = T(\eps) > 0$ be so that (\ref{eq_2136}) holds true. 
	Let $x \in M$, $t_{0} \geq T$ and suppose that $u_{0},u_{1} \in \RR^{n}$ are such that
	$ |u_{0} - u_{1}| \leq \sqrt{n} $.
	Then $\delta(x,u_{0},t) < \delta(x, u_{1}, t_{0}) + 2\eps$
	for all $t \geq t_{0} + T$.
	\label{lem_1458}
\end{lemma}

\begin{proof}
	Since $|u_{0}-u_{1}| \leq \sqrt{n}$ and $t_{0} \geq T$, for any $t \geq t_{0} + T$ we may apply (\ref{eq_2136}) to the triangle with vertices $\tilde{v}(u_{1},- t_{0})$, $\tilde{v}(u_{1},t_{0})$ and $\tilde{v}(u_{0},t)$, so that 
	$$
	d(\tilde{v}(u_{1}, t_{0}), \tilde{v}(u_{1}, -t_{0})) +
	d(\tilde{v}(u_{1}, t_{0}), \tilde{v}(u_{0}, t)) - 
	d(\tilde{v}(u_{0}, t), \tilde{v}(u_{1}, -t_{0})) < \eps.
	$$
	Similarly we may use (\ref{eq_2136}) for the triangle with vertices $\tilde{v}(u_{0},-t)$, $\tilde{v}(u_{1}, -t_{0})$ and $\tilde{v}(u_{0},t)$, so that
	$$
	d(\tilde{v}(u_{1}, -t_{0}), \tilde{v}(u_{0}, t)) + 
	d(\tilde{v}(u_{1}, -t_{0}), \tilde{v}(u_{0}, -t)) - 
	d(\tilde{v}(u_{0}, t), \tilde{v}(u_{0}, -t)) < \eps.
	$$
	Summing these two inequalities, for any $t \geq t_{0} + T$ we have that
	\begin{align}
		d(\tilde{v}(u_{1}, t_{0}), \tilde{v}(u_{1}, -t_{0})) &+
		d(\tilde{v}(u_{1}, t_{0}), \tilde{v}(u_{0}, t)) 
		\nonumber \\ &+
		d(\tilde{v}(u_{1}, -t_{0}), \tilde{v}(u_{0}, -t)) < 
		d(\tilde{v}(u_{0}, t), \tilde{v}(u_{0}, -t)) + 2\eps.
		\label{eq_1910}
	\end{align}
	Using (\ref{eq_1910}) together with the triangle inequality we obtain that
	\begin{align*}
		\delta(x,u_{0},t) &= 
		d(x, \tilde{v}(u_{0}, -t)) + 
		d(x, \tilde{v}(u_{0}, t)) - 
		d(\tilde{v}(u_{0}, t), \tilde{v}(u_{0}, - t)) 
		\\ & <
		d(x, \tilde{v}(u_{0}, -t)) + 
		d(x, \tilde{v}(u_{0}, t)) 
		\\ & - d(\tilde{v}(u_{0}, -t), \tilde{v}(u_{1}, -t_{0})) -
		d(\tilde{v}(u_{1}, t_{0}), \tilde{v}(u_{0}, t)) -
		d(\tilde{v}(u_{1}, -t_{0}), \tilde{v}(u_{1}, t_{0})) + 2 \eps
		\\ & \leq
		d(x, \tilde{v}(u_{1},-t_{0})) +
		d(x, \tilde{v}(u_{1},t_{0})) - 
		d(\tilde{v}(u_{1}, -t_{0}), \tilde{v}(u_{1}, t_{0})) + 2\eps 
		\\ &= 
		\delta(x,u_{1},t_{0}) + 2\eps,
	\end{align*}
	and the proof is completed.
\end{proof}

\begin{figure}
	\begin{center} \includegraphics[width=6in]{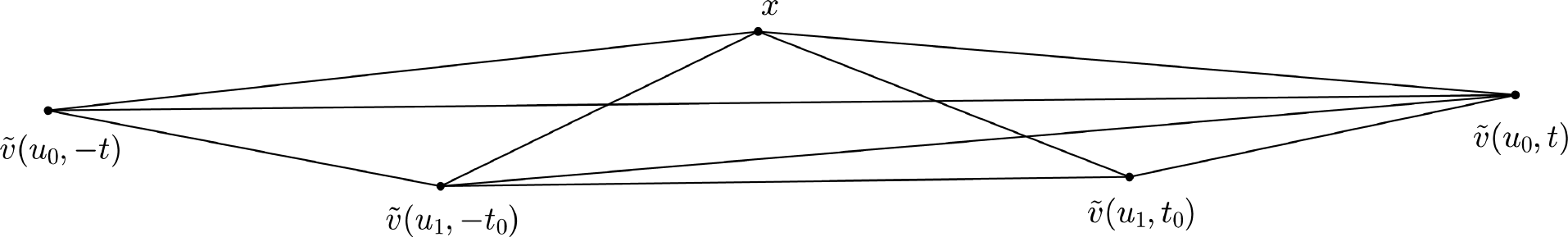} \end{center}
	\caption{An illustration of the triangles involved in the proof of Lemma \ref{lem_1458}. Note that the edges  represent geodesic segments in $M$ that do not necessarily intersect.
	} 
\end{figure}

\begin{corollary}
	Let $\eps > 0$ and $T = T(\eps) > 0$ be so that (\ref{eq_2136}) holds true.
	Then for any $x \in M$, $u_{0} \in \RR^{n}$ and $t_{0} \geq T$ we have that $f_{v}(x) \leq \delta(x,u_{0},t_{0}) +2 \eps$.
	\label{cor_1933}
\end{corollary}

\begin{proof}
	Using Lemma \ref{lem_1458} we have that $\delta(x,u_{0},t) < \delta(x,u_{0},t_{0}) +2 \eps$ for all $t \geq t_{0} + T$. Therefore according to Lemma \ref{lem_1930} we obtain that $f_{v}(x) = \lim_{t \to \infty} \delta(x,u_{0},t) \leq \delta(x,u_{0},t_{0}) + 2\eps$.
\end{proof}

\section{No conjugate points}
\label{sec_ncp}

For the topological part of the proof we again construct a continuous map $\Phi$ from $\RR^{n}$ onto the manifold $M$.
The construction is quite different in essence from the one described in Section \ref{section_phi} but relies on a similar framework.
Let $L = \ZZ^{n}$, and for a point $p \in L$ write $Q(p) = p + [0,1]^{n} \subseteq \RR^{n}$ for the cube of side-length $1$, and $\text{ver}(Q(p)) \subseteq L$ for its set of vertices.

\begin{proposition}
	For any $v \in S^{n-1}$ and $\eps > 0$ there exists a continuous map $\Phi : \RR^{n} \to M$ for which the following hold:
	\begin{enumerate}
		\item[(a)] There exists a constant $C \geq 0$ depending on $n$, $\delta$ and $\eps$ such that 
		$$
		d(\Phi(x), \tilde{p}) \leq C
		\qquad \text{for any } p \in L \text{ and } x \in Q(p).
		$$
		\item[(b)] For any $x \in \RR^{n}$ we have that $f_{v}(\Phi(x)) \leq (1+2n) \cdot \eps$.
	\end{enumerate}
	\label{prop_1930}
\end{proposition}

\begin{proof}
Fix $v \in S^{n-1}$, $\eps > 0$ and let $T = T(\eps) > 0$ be so that (\ref{eq_2136}) holds true.  
We say that the affinely independent points $\{p_{1},...,p_{k}\}$ are \textit{vertices of a unit cube} if they belong to $\text{ver}(Q(p))$ for some $p \in L$.
A $k$-simplex in the triangulation of $Q(p) \subseteq \RR^{n}$ is the convex hull of $k+1$ affinely independent vertices of a unit cube.   
We construct a ``simplex mapping'' $\Delta$ that for any $0 \leq k \leq n$ and $k+1$ affinely independent vertices of a unit cube $u_{0},u_{1},...,u_{k} \in L$ assigns a continuous map
$$
\Delta[u_{0},u_{1},...,u_{k}] : \text{conv}(u_{0},u_{1},...,u_{k}) \to M,
$$
	such that the following hold:
	\begin{enumerate}
		\item[(i)] For $k = 0$ we have $\Delta[u_{0}] = \nu$, i.e. $\Delta[u_{0}](u_{0}) = \tilde{u}_{0}$.
		\item[(ii)] The mapping $\Delta$ is consistent with the simplicial complex structure, i.e. for any $1 \leq k \leq n$ and affinely independent vertices of a unit cube $u_{0},u_{1},...,u_{k} \in L$ we have that
		$$
		\Delta[u_{0},u_{1},...,u_{k}] \big|_{\text{conv}(J(u_{0},u_{1},...,u_{k}))} = \Delta[J(u_{0},u_{1},...,u_{k})]
		\qquad
		\text{for any } J \subseteq \{0,1,...,k\}.
		$$
		\item[(iii)] For any $1 \leq k \leq n$ and affinely independent vertices of a unit cube $u_{0},u_{1},...,u_{k} \in L$,
		$$
		\delta(y, u_{0}, k \cdot T) < \left( 1 + 2(k-1)\right) \cdot \eps
		\qquad \text{for any } y \in \Im(\Delta[u_{0},u_{1},...,u_{k}]).
		$$ 
		\item[(iv)] For any $1 \leq k \leq n$ and affinely independent vertices of a unit cube $u_{0},u_{1},...,u_{k} \in L$,
		$$
		d(y, \tilde{v}(u_{0}, k \cdot T)) \leq 
		k \cdot \left( \sqrt{n} + T + 2\delta\right)
		\qquad \text{for any } y \in \Im(\Delta[u_{0},u_{1},...,u_{k}]).
		$$
	\end{enumerate}
	
	The definition of $\Delta$ is done by induction on $k$. For a single point $u_{0} \in L$ we define $\Delta[u_{0}]$ as required by (i). 
	For $k = 1$, let $u_{0}, u_{1} \in L$ be two distinct vertices of a unit cube. 
	Write $\Delta = \Delta[u_{0},u_{1}] : \text{conv}(u_{0},u_{1}) \to M$ for the map to be defined. For $u_{0}$ and $u_{1}$ we define $\Delta$ to be $\nu$, so that 
	$$
	\Delta(u_{j}) = \Delta[u_{j}](u_{j}) = \tilde{u}_{j}
	\qquad \text{for } j \in \{0,1\}.
	$$ 
	In particular we obtain that (ii) holds true.
	Let $\xi $ be the midpoint of the line segment $\text{conv}(u_{0},u_{1})$, and map it by $\Delta(\xi) = \tilde{v}(u_{0},T) \in X$.
	Define the map $\Delta$ on each line segment in $\RR^{n}$ connecting $\xi$ and $u_{j}$ to be the minimizing geodesic segment $\eta_{j}$ connecting $\Delta(\xi) = \tilde{v}(u_{0},T)$ to $\Delta(u_{j}) = \tilde{u}_{j}$. 
	The map $\Delta$ is clearly continuous. 
	Since $u_{0}$ and $u_{1}$ are vertices of a unit cube we have that $|u_{1} - u_{0}| \leq \sqrt{n}$. Hence according to Lemma \ref{lem_1522} together with (\ref{eq_2136}), for any $j \in \{0,1\}$ and $y \in \eta_{j}$ we have 
	$$
	\delta(y,u_{0},T) \leq \delta(\tilde{u}_{j}, u_{0}, T) < \eps,
	$$
	so that (iii) holds true. 
	Moreover, using (\ref{eq_1651}) together with the triangle inequality we have that 
	$$
	d(\tilde{u}_{0}, \tilde{v}(u_{0},T)) < T + 2\delta 
	\qquad \text{and} \qquad
	d(\tilde{u}_{1}, \tilde{v}(u_{0}, T)) < \sqrt{n} + T + 2\delta.
	$$
	Since $\eta_{j}$ is defined to be the minimizing geodesic connecting $\tilde{v}(u_{0},T)$ and $\tilde{u}_{j}$, we obtain that (iv) holds true as well.
	
	\medskip
	The construction for $k \geq 2$ is done in a similar way. Let $u_{0},u_{1},...,u_{k} \in L$ be affinely independent vertices of a unit cube.
	By the induction hypothesis we have a continuous map 
	$$
	\partial \Delta : \partial\; \text{conv}(u_{0},u_{1},...,u_{k}) \to M,
	$$
	obtained by gluing all of the maps $\{\Delta[u_{0},...,\hat{u}_{j},...,u_{k}] : 0 \leq j \leq k\}$, where $\hat{u}_{j}$ stands for the omission of the element with index $j$ from the expression.
	The gluing is valid thanks to property (ii) of the induction hypothesis.
	Abbreviate $\sigma = \text{conv}(u_{0},...,u_{k})$, $t_{0} = (k-1) \cdot T$ and $\Delta = \Delta[u_{0},...,u_{k}]$ for the map to be defined.
	On the boundary of the simplex we define $\Delta |_{\partial \sigma} = \partial \Delta$ so that (ii) holds true.
	In order to extend the map $\Delta$ to the interior of the simplex $\sigma$, we first define it on the centroid $\xi \in \text{int} (\sigma)$ by $\Delta(\xi) = \tilde{v}(u_{0}, t_{0} + T) \in X$.
	Any line in $\sigma$ connecting $\xi$ to a point $\zeta \in \partial \sigma$ is then mapped to the geodesic segment connecting $\Delta(\xi) = \tilde{v}(u_{0}, t_{0} + T)$ and $\Delta(\zeta)$.
	Since $\partial \Delta$ is continuous and the exponential map is a diffeomorphism at points of $X$ by \cite[Corollary 4.1]{EK}, the map $\Delta$ is continuous.
	
	\medskip
	Let $\zeta \in \partial \sigma$, write $z = \Delta(\zeta)$, and let $\eta$ be the minimizing geodesic connecting $\tilde{v}(u_{0}, t_{0} + T)$ and $z$.
	Using the induction hypothesis, it follows from (iii) that $\delta(z, u_{j}, t_{0}) < \left( 1 + 2(k-2) \right) \cdot \eps$ for either $j = 0$ or $j = 1$.
	Since $|u_{0} - u_{1}| \leq \sqrt{n}$, as they are both vertices of a unit cube, using Lemma \ref{lem_1522} and Lemma \ref{lem_1458} we have that for $j \in \{0,1\}$ and any $y \in \eta$,
	$$
	\delta(y, u_{0}, t_{0} + T) \leq \delta(z, u_{0}, t_{0} + T) < \delta(z, u_{j}, t_{0}) + 2\eps < \left( 1 + 2(k-1)\right) \cdot \eps.
	$$
	Thus we see that (iii) holds true.
	Moreover, using the induction hypothesis, it follows from (iv) that $d(z, \tilde{v}(u_{j}, t_{0})) \leq (k-1) \cdot \left( \sqrt{n} + T + 2\delta \right)$. Therefore by (\ref{eq_1651}) together with the fact that $|u_{0}-u_{1}| \leq \sqrt{n}$ and $y$ lies on the geodesic segment connecting $z$ and $\tilde{v}(u_{0}, t_{0} + T)$, we obtain that
	\begin{align*}
		d(y, \tilde{v}(u_{0}, k \cdot T)) &= 
		d(y, \tilde{v}(u_{0}, t_{0} + T)) \leq 
		d(z, \tilde{v}(u_{0}, t_{0} + T)) 
		\\ &\leq 
		d(z, \tilde{v}(u_{j}, t_{0})) + d(\tilde{v}(u_{j}, t_{0}), \tilde{v}(u_{0}, t_{0} + T)) \leq  k \cdot \left( \sqrt{n} + T + 2\delta \right).
	\end{align*}
	Hence (iv) holds true as well, and the construction of $\Delta$ is completed.
	
\medskip
In order to obtain the map $\Phi : \RR^{n} \to M$ described in the proposition, we glue together $\Delta$-maps as we did in Section \ref{subsec_gluing}.
This is again possible thanks to property (ii).
Moreover, using the simplex notation described in (\ref{eq_1642}), according to property (iv) together with (\ref{eq_1651}) and the triangle inequality, for any $p \in L$ and $\pi \in S_{n}$ we have
	\begin{align*}
		\sup \{ d (\tilde{p}, y) : 
		y \in \Im(\Delta[\sigma(p,\pi)])\} &\leq 
		\sup \{ d (y, \tilde{v}(p, n \cdot T)) : 
		y \in \Im(\Delta[\sigma(p,\pi)])\} 
		\\ &+ d(\tilde{p}, \tilde{v}(p, n \cdot T)) 
		\leq n \cdot \left( \sqrt{n} + T + 2\delta \right) + n \cdot T + 2\delta.
	\end{align*}
Thus by gluing together all the maps $\{\Delta[\sigma(p, \pi)] : p \in L, \; \pi \in S_{n}\}$ we obtain a continuous map $\Phi : \RR^{n} \to M$ for which property (a) of the proposition holds true.
Moreover, property (iii) of the $\Delta$-maps implies that for any $x \in Q(p)$ we have
$ \delta(\Phi(x), p, n \cdot T) < (1 + 2(n-1)) \cdot \eps $.
It therefore follows from Corollary \ref{cor_1933} that 
$$
f_{v}(\Phi(x)) \leq (1+2n) \cdot \eps
\qquad \text{for any } x \in \RR^{n}.
$$
Hence property (b) of the map $\Phi$ holds true as well, thus completing our construction.
	
\end{proof}

\medskip
Let us now complete the proof of Proposition \ref{prop_1530}, and consequently show that Theorem \ref{thm_1645} holds true.
Fix $v \in S^{n-1}$ and $\eps > 0$. Let $\Phi: \RR^{n} \to M$ be the continuous map constructed in Proposition \ref{prop_1930}. 
Following the argument given in Section \ref{subsec_surjective} we see that property (a) of the proposition implies that $\Phi$ is surjective. 
Therefore by property (b) of the proposition we obtain that $f_{v}(y) \leq (1 + 2n) \cdot \eps$ for any $y \in M$. 
Since $\eps$ can be taken to be arbitrarily small, we obtain that $f_{v} \equiv 0$ as required.

\section{Large-scale geometry}
\label{sec_largescale}

For $x \in M$ and $v \in S^{n-1}$ we write $\gamma_{x,v}: \RR \to M$ for the transport line of $B_{v}$ with $\gamma_{x,v}(0) = x$.
The existence of this transport line is assured by the fact that the function $B_{v}$ foliates, and $\gamma_{x,v}$ is a complete, minimizing geodesic.
We think of $\gamma_{x,v}$ as the geodesic emanating from the point $x$ in a ``global direction'' $v$.
Another natural way to associate a direction with the geodesic $\gamma_{x,v}$ at the point $x$ is the unit tangent vector $\dot{\gamma}_{x,v}(0) \in S_{x}M = \{ u \in T_{x}M \; ; \; \Vert u \Vert = 1\}$, to which we refer as the ``local direction''.
It is therefore natural to inquire about the map 
$$
S^{n-1} \ni v \mapsto \dot{\gamma}_{x,v}(0) \in S_{x}M,
$$ 
that maps a global direction to a local direction.
Recall from (\ref{eq_1740}) that $\dot{\gamma}_{x,v}(0) = \nabla B_{v}(x)$ as we defined $\gamma_{x,v}$ to be a transport line of $B_{v}$.
For the proof of Theorem \ref{thm_1645} we showed that the map $v \mapsto \nabla B_{v}(x)$ is continuous and onto, which is the result of Proposition \ref{prop_1530} and Lemma \ref{lem_1841}.
Using Theorem \ref{thm_1730} we will now show that this map is in fact a homeomorphism. 

\medskip
Let $x \in M$, $v \in S^{n-1}$ and abbreviate $\gamma = \gamma_{x,v}$. According to Theorem \ref{thm_1730}, for any $t \geq 0$ there exists $p_{t} \in X$ such that $d(p_{t},\gamma(t)) < 2 \delta n$. Using the triangle inequality we see that 
\begin{equation}
| |p_{t}| - t| = |d(p_{t},0) - d(\gamma(t), \gamma(0))| \leq 
d(p_{t}, \gamma(t)) + d(0,x) < 2\delta n + d(0,x).
\label{eq_2331}
\end{equation}
Moreover, since $\gamma$ is a transport line of the $1$-Lipschitz function $B_{v}$ we have that
\begin{align}
|p_{t}| - \langle p_{t}, v \rangle = 
|p_{t}| - B_{v}(p_{t}) &\leq 
|p_{t}| - B_{v}(\gamma(t)) + 2 \delta n 
\nonumber \\ &= 
|p_{t}| - t - B_{v}(x) + 2 \delta n < 
2 \left( d(0,x) + 2 \delta n \right).
\label{eq_2332}
\end{align}
Using (\ref{eq_2331}) we see that $|p_{t}|$ tends to infinity with $t$, and therefore it follows from (\ref{eq_2332}) that $p_{t} \rightsquigarrow v$.
Hence we see that the global direction $v$ is determined by the transport line $\gamma$.
In other words, the map $v \mapsto \nabla B_{v}(x)$ is one-to-one. As mentioned above, in previous sections we showed that this map is continuous and onto.  Therefore we conclude the following:

\begin{corollary}
For any $x \in M$, the map $v \mapsto \nabla B_{v}(x)$ from $S^{n-1}$ to $S_{x}M$ is a homeomorphism.
\label{cor_61}
\end{corollary}

\medskip
Using the approximations above we also obtain our next corollary, which states that the large-scale geometry of the manifold $M$ tends to the Euclidean one. According to (\ref{eq_2331}) and (\ref{eq_2332}) we see that if $t > C := 3 \left( d(0,x) + 2\delta n \right)$ then
$$
|p_{t} - tv|^{2} = |p_{t}|^{2} - 2t \langle p_{t},v\rangle + t^{2} < 
(t + C)^{2} - 2t \left( t - C \right) + t^{2} = 4Ct + C^{2}.
$$
In case $t \leq C$ we have $|p_{t} - tv| \leq |p_{t}| + t \leq 3C$. 
Hence there exist constants $C_{1},C_{2} > 0$ depending on $n$, $\delta$ and $d(0,x)$ such that
$$
|p_{t} - tv| \leq C_{1}  \sqrt{t} + C_{2}. 
$$
For $w \in S^{n-1}$ and $s \geq 0$ we may similarly let $q_{s} \in X$ be such that $d(q_{s}, \gamma_{x,w}(s)) < 2 \delta n$ and obtain that $|q_{s} - sw| \leq C_{1} \sqrt{s} + C_{2}$. It therefore follows from the triangle inequality that 
$$
\left| d(\gamma_{x,v}(t), \gamma_{x,w}(s)) - |tv - sw| \right| \leq 
C_{1} (\sqrt{t} +  \sqrt{s}) + 2 C_{2} + 4 \delta n.
$$
Using this approximation we see that if we fix $t,s \geq 0$ and write $a = tv$, $b = sw$, then 
$$
d(\gamma_{x,v}(tr), \gamma_{x,w}(sr)) = r \cdot |a-b| + O(\sqrt{r})
\qquad \text{as } r \to \infty.
$$
In particular we obtain the following:

\begin{corollary}
Let $x \in M$.
Then for $a, b \in \RR^{n}$, writing $a = tv, b = sw$ with $v,w \in S^{n-1}$ and $t,s \geq 0$, we have
$$
\lim_{r \to \infty} 
\frac{d(\gamma_{x,v}(tr), \gamma_{x,w}(sr))}{r} = 
|a-b|,
$$
and the convergence is locally uniform in $a,b \in \RR^{n}$.
\end{corollary}

\medskip
\noindent Department of Mathematics, Weizmann Institute of Science, Rehovot 76100, Israel. \\
{\it e-mail:} \verb"matan.eilat@weizmann.ac.il"

\end{document}